\documentclass[12pt,a4paper]{amsart}

\usepackage{amsfonts}
\usepackage{amssymb}
\usepackage{amsthm}
\usepackage{amsmath}
\usepackage{enumerate}
\usepackage{hyperref}
\usepackage{mathrsfs}

\usepackage{graphicx}
\usepackage{datetime}

\newcommand{\IC}{\mathbb{C}}

\newcommand{\IN}{\mathbb{N}}
\newcommand{\IK}{\mathbb{K}}

\newcommand{\BH}{\mathcal{B}(H)}
\newcommand{\BX}{\mathcal{B}(X)}
\newcommand{\HCX}{\mathcal{HC}(X)}
\newcommand{\HCH}{\mathcal{HC}(H)}
\newcommand{\MiX}{\mathcal{M}\mathrm{ix}(X)}

\newcommand{\SPAN}{\mathop{\mathrm{span}}}

\newtheorem{prop}{Proposition}[section]
\newtheorem{lem}[prop]{Lemma}
\newtheorem{thm}[prop]{Theorem}
\newtheorem{cor}[prop]{Corollary}
\theoremstyle{definition}
\newtheorem{df}[prop]{Definition}
\theoremstyle{remark}

\newtheorem*{claim}{Claim}

\numberwithin{equation}{section}

\begin{document}

\title{Operators approximable by hypercyclic operators}
\author{James Boland}
\address{School of Mathematics\\Trinity College\\Dublin 2}
\email{jaboland@tcd.ie}
\thanks{This work was supported
by the Science Foundation Ireland under grant 11/RFP/MTH3187.}

\date{\currenttime{} \today}

\begin{abstract}
	We show that operators on a separable infinite dimensional
	Banach space $X$ of the form $I +S$, where $S$ is an operator with dense generalised kernel, 
	must lie in the norm closure of the hypercyclic operators on
	$X$, in fact in the closure of the mixing operators.
\end{abstract}

\maketitle

\section{Introduction}

Hypercyclic operators $T \colon X \to X$ on a Banach space $X$ (those
bounded linear operators
for which there is a vector $v \in X$ with dense orbit $\{ v, Tv, T^2v,
\ldots\}$) have been the subject of intense study in recent decades and
we refer to \cite{BayartMatheronBook} for a recent overview of the
topic.
Many problems that were posed about hypercyclicity have now been solved.
In particular, it is well
known that every separable infinite dimensional
Banach space $X$ supports hypercyclic operators and in fact many such
operators.
We use $\BX$ for the bounded linear operators on $X$ (always over $\IC$)
and $\HCX$ for the
subclass of hypercyclic operators.
An elegant result of \cite{HadwinNRR1979} allows one to conclude that,
for $T \in \HCX$,
the similarity class $\{ S^{-1} T S : S \in \BX \mbox{ invertible}\}$ is
always dense in 
the strong operator topology
of $\BX$ 
(and consists of hypercyclic operators).

In the operator norm on $\BX$, there are some obvious restrictions such
as $\|T\| > 1$ for $T \in \HCX$, but
Herrero
\cite{Herrero1979JFA} characterised the norm closure $\overline{\HCH}$
of $\HCH$
in spectral terms
for $H$ a (separable infinite dimensional) Hilbert space.
In fact the paper \cite{Herrero1979JFA} may be seen as flowing
from a considerable body of work (see \cite{HerreroBookI,AFHVBookII})
aimed at characterising the norm closure of the similarity class of a
Hilbert space operator, but \cite[Remark 2.3]{Herrero1979JFA} notes that
the necessary and sufficient conditions for a Hilbert space operator $T
\in \BH$ to
belong in $\overline{\HCH}$
are
also necessary for the Banach space case.
In the extreme case of pathological Banach spaces $Y$ with few operators,
the existence of which was recently established by Argyros and Haydon
\cite{ArgyrosHaydon2011}, all operators $T \in \mathcal{B}(Y)$ are of
the form $T = \lambda I + K$ with $\lambda$ a scalar and $K$ compact,
and then the Herero conditions reduce to $|\lambda| = 1$ and $K$
quasinilpotent (spectrum just $\{0\}$).

Our main result (Theorem ~\ref{thm:genker}) implies (Corollary ~\ref{cor:nilpapprox}) that (for $X$ a separable
infinite dimensional Banach space) operators $T = \lambda I + N$ with
$|\lambda| = 1$ and $N\in\BX$ nilpotent are all in the operator norm
closure of $\HCX$. 
So far we have been unable to decide whether this result remains true with $N$ replaced by a quasinilpotent $K$, 
but we have some results for special cases of quasinilpotent $K$.

The methods we use to obtain this approximation are largely elementary,
based on certain notions of generalised shifts --- which may be seen as
operators that admit a Jordan form representation in a somewhat loose
sense. We rely then on an observation of V. M\"uller
\cite[Theorem 1]{Muller2010IEOT} which is a simplification of \cite[Theorem 2.2]{BayartMatheronBook}, 
attributed to Grivaux and Shkarin, to
show that translates of such shifts by the identity operator $I$, or by
$\lambda I$ with $|\lambda| = 1$, are mixing (and hence hypercyclic).

We recall that $T \in \BX$ is called a \emph{mixing operator} if for
each pair $U, V$ of nonempty open subsets of $X$, there is $N$ such that
$T^n(U) \cap V \neq \emptyset \forall n > N$. We will use $\MiX$ to
denote the mixing operators on $X$.

\section{Generalised shifts}

For $T \in \BX$ we use the notation $\mathcal{N}^\infty(T)$ for the generalised
kernel $\bigcup_{n=1}^\infty \ker (T^n)$ and $\mathcal{R}^\infty(T) =
\bigcap_{n=1}^\infty T^n(X)$ for the hyperrange.

It will be convenient for us to introduce
the (lexicographic)
ordering $<$ on $\IN\times \IN$ given by
$(p.q)<(\ell,k)$ if either $p<l$ or $p=\ell$ and $q<k$.
Given $(\ell,k)\in \IN\times\IN$ with $k>1$ we refer to $(\ell,k-1)$
as the \emph{immediate predecessor} of $(\ell,k)$.
This helps to explain the notions of generalised (backward) shift that
we will use. Notice that we base these notions of shifts
on sets of vectors which span a
dense linear subspace, but need not be Schauder bases for $X$.

\begin{df}
	\label{def:genshift}
Let $X$ be a separable infinite dimensional Banach space.
Given a set of linearly independent vectors $\{x^\ell_k:(\ell,k)\in \{1,\ldots,n\}\times \IN\}$ whose span is dense in $X$,
we call $S\in \BX$ a \emph{generalised backward $n$-shift adapted to the set}
if for each $(\ell,k)\in \{1,\ldots,n\}\times \IN$,
$T(x^\ell_k)$ is a finite linear combination of $x^p_q$ with
$(p,q)<(\ell,k)$ and is such that 
the coefficient of the immediate predecessor, $x^\ell_{k-1}$, is non-zero
when $k>1$. In more formal terms, we insist that
\begin{equation}
	\label{eqn:genshiftdef}
	\begin{split}
	Sx^\ell_k=\sum^{k-1}_{i=1}a^{\ell,k}_ix^\ell_i
	+\sum_{i=1}^{\ell-1}\sum^{m^{\ell,k}_j}_{j=1}b^{\ell,k}_{i,j}x^i_j
	\qquad
	a^{\ell,k}_i, b^{\ell,k}_{i,j}\in\IK,\\
	a^{\ell,k}_{k-1}\neq 0 \mbox{ if } k > 1,
	m^{\ell,k}_j\in\IN.
\end{split}
\end{equation}

Similarly, given a set of linearly independent vectors $\{x^\ell_k:(\ell,k)\in \IN\times \IN\}$ whose span is dense in $X$,
we call $S\in \BX$ a \emph{generalised backward $\infty$-shift (adapted to
the set)}
if for each $(\ell,k)\in \IN\times \IN$,
$T(x^\ell_k)$ is a finite linear combination (\ref{eqn:genshiftdef}).
 
We think of the sequence $(x^\ell_k)_{k=1}^\infty$
as the $\ell^{\mathrm th}$ chain of the set.
\end{df}

\begin{lem}
\label{lem:bnslem}
If $S\in \BX$ is a
generalised $n$-shift or $\infty$-shift adapted to
a set $\{x^\ell _k\}$,
then $x^\ell_k\in \mathcal{R}^\infty (S)\cap \mathcal{N}^\infty (S)$ (for each $(\ell, k)$
in the set).
\end{lem}

\begin{proof}
The proof is a simple induction on $\ell$.

For $\ell=1$, $Sx^1_k=\sum^{k-1}_{i=1}a^{1,k}_ix^1_i \in \SPAN\{x^1_i :
1 \leq i \leq k -1\}$, $x^1_1 \in \ker
S$. Clearly $x^1_k \in \ker S^k \subseteq \mathcal{N}^\infty(S)$ for each $k$. 
Also, 
\[
	S^r\left(\frac{x^1_{\ell+1}}{\prod_{i=1}^{r}a^{1,i+1}_i}\right)
	=x^1_1
\]
and so $x^1_1\in S^r(X)$ for all $r$. Hence $\IC x^1_1\subseteq
\mathcal{R}^\infty (S)$.
Using a subinduction on $k$,
assume
$x^1_i\in \mathcal{R}^\infty (S)$ for $1\le i \le k$.
Then
\[
	S^r\left(\frac{x^1_{r+k+1}}{\prod_{i=1}^{r}a^{1,i+k+1}_{i+k}}\right)=x^1_{k+1}+u
\]
with $u\in \SPAN \{ x^1_i : 1\le i \le k\} \subseteq
\mathcal{R}^\infty (S)$
and $u \in S^r(X) \Rightarrow x^1_{k+1} \in S^r(X)$ for $r \geq 1$. Hence
$x^1_{k+1} \in \mathcal{R}^\infty(S)$. By induction we get $x^1_k \in \mathcal{R}^\infty(S)$
for all $k$.

Similar arguments (using finiteness of the sums
(\ref{eqn:genshiftdef}))
allow us to show the induction step on $\ell$.
\end{proof}

The following is a simplified version of \cite[Theorem
1]{Muller2010IEOT} which is a simplification of \cite[Theorem 2.2]{BayartMatheronBook}.

\begin{thm}
\label{thm:muller}
Let $X$ be a separable Banach space, $S\in \BX$ with
$\mathcal{N}^\infty(S)\cap \mathcal{R}^\infty(S)$ dense in $X$ and let $\lambda \in \IC$
with $|\lambda|=1$. Then $\lambda+S \in \MiX$.
\end{thm} 

\begin{cor}
\label{cor:bns}
If $X$ is a Banach space,
$S\in \BX$ a generalised $n$-shift or $\infty$-shift (adapted to
some set  $\{x^\ell_k\}$) and if $\lambda \in \IC$ with $|\lambda|=1$, then 
$\lambda +S \in \MiX$.
\end{cor}

\begin{proof}
This follows from Theorem~\ref{thm:muller} and Lemma~\ref{lem:bnslem}
\end{proof}

\section{Approximating operators with dense generalised kernel}

Our aim now is to show that if an operator $S \in \BX$ (with
$X$ as usual a separable infinite dimensional Banach space) has dense generalised kernel 
then it can be
approximated in norm by generalised shifts (as defined in
Definition~\ref{def:genshift}). Our approach is to first find linearly
independent vectors
\[
	\{ y^\ell_j : 1 \leq \ell < \infty, 1 \leq j \leq d_\ell \}
\]
with dense linear span in $X$ 
which are well adapted to $S$ in a sense
somewhat similar to a forward shift version of
(\ref{eqn:genshiftdef}). There will however be some
important differences between what we can achieve for general $S$ and what
is required in (\ref{eqn:genshiftdef}). 
In particular, the 
vectors $y^\ell_j$ will not necessarily belong to $\mathcal{R}^\infty(S)$ and our approach will be to adjust the operator $S$ (with
an adjustment of at most a prescribed $\varepsilon> 0$ in norm) so as
to be able to apply Corollary~\ref{cor:bns}.

For any sequence or set of vectors $\{x_n\}$ we will use square brackets
$[x_n]$ for the closure of their linear span. 

We describe our construction as an algorithm, starting from a fixed
sequence $\{x_n : n \in \IN\}\subset \mathcal{N}^\infty(S)$ so that $[x_n : n \in \IN] = X$ (which
certainly exists since $X$ is assumed to be separable). We suppose also
that $\{x_n : n \in \IN\}$ is algebraicly linearly independent (or at
least that $x_1 \neq 0$).

We begin by setting
\begin{equation}
	\label{eqn:y-1-1}
	y^1_1 = x_1.
\end{equation}
Then there is a smallest $r \geq 1$  with 
\[
	S^r y^1_1 =0.
\]
and we choose $d_1 = r$. Then
the space $E_1 = \SPAN\{x_1, S x_1, \ldots, S^{r-1} x_1\}$ is
$S$-invariant and
we let
\begin{equation}
	\label{eqn:y-1-j}
	y^1_{j} = S^{j-1} y^1_1 \quad (2 \leq j \le  d_1 ).
\end{equation}

Assuming that we have already chosen $d_1, \ldots, d_\ell \in \IN $ and vectors
$\{y^i_j : 1 \leq i \leq \ell, 1 \leq j < d_i +1\}$, we will use the
notation
\begin{equation}
	\label{eqn:E-m}
	E_m = [ y^i_j : 1 \leq i \leq m, 1 \leq j \le d_i]
\end{equation}
(even for $m=0$ when $E_m = \{0\}$)
and insist during the algorithm
that $S(E_m) \subseteq E_m$ holds for $m =1, 2, \ldots, \ell$
as well as
\begin{equation}
	\label{eqn:shift-cond}
	S y^{m}_j = y^{m}_{j+1}
	\qquad (1 \leq m \leq \ell, 1 \leq j < d_{m}),
\end{equation}
\begin{equation}
	\label{eqn:shift-cond2}
	S y^{m}_{d_m}  \in E_{m-1}
	\qquad (1 \leq m \leq \ell),
\end{equation}
and
\begin{equation}
	\label{eqn:ind-cond}
	\SPAN\{ y^{m}_j : 1 \leq j \leq d_{m} \} \cap E_{m-1} = \{0\}
	\qquad (1 \leq m \leq \ell).
\end{equation}
We will have in addition that 
\begin{equation}
	\label{eqn:useupxns}
	\{ x_1, \ldots, x_m\} \subset E_m \quad (1 \leq m \leq \ell),
\end{equation}
\begin{equation}
	\label{eqn:use-only-xns}
	y^m_1 \in \{ x_1, x_2, \ldots\} \quad (1 \leq m \leq \ell)
\end{equation}
and
\begin{equation}
	\label{eqn:ys-lin-indep}
	\{ y^m_j : 1 \leq i \leq \ell , 1 \leq j  \leq d_m \} \mbox{
	linearly independent}.
\end{equation}
Notice that the construction so far
has ensured that these conditions hold initially for $\ell =1$.

The next step is to choose the smallest $n$ such that $x_n \notin E_\ell$
and to put $y^{\ell+1}_1  = x_n$ (which will ensure that
(\ref{eqn:useupxns}) and (\ref{eqn:use-only-xns} hold for $m = \ell +1$)
and then there is a smallest $r \geq 1$ such that
\begin{equation}
	\label{eqn:infinitecase}
	S^{r+1} y^{\ell+1}_1 \in \SPAN \{ y^{\ell+1}_1, S
		y^{\ell+1}_1, S , \ldots ,
		        S^r y^{\ell+1}_1 \} + E_\ell
			\qquad (\forall r \in \IN),
\end{equation}
We put $d_{\ell+1} = r$. Then we take
\[
	y^{\ell+1}_{j+1} =  S^j y^{\ell+1}_1 \qquad (1 \leq j < d_{\ell +1}).
\]
It is clear then that
(\ref{eqn:shift-cond}) and
(\ref{eqn:ind-cond}) hold for $m = \ell+1$. 
To show (\ref{eqn:shift-cond2}), we can consider the operator induced by $S$ on the quotient
$X/E_\ell$, which we temporarily denote $\tilde S$. Writing $\tilde x$
for the coset $x + E_\ell$ we then have $\tilde S (\tilde x) =
\widetilde{Sx}$ and $\tilde S$ is a nilpotent operator with an
invariant subspace $\SPAN \{ \widetilde { y^{l+1}_j} : 1 \leq j \leq
r\}$ of dimension $r$. Hence ${\tilde S}^r = 0$ on this subspace
and
(\ref{eqn:shift-cond2}) holds for $m = \ell+1$, so that also $T(E_{\ell+1})
\subseteq E_{\ell+1}$.
Finally, (\ref{eqn:ys-lin-indep}) is easy to check for $m=\ell +1$.

\begin{thm}
\label{thm:genker}
Let $X$ be a separable infinite dimensional Banach space and let $S\in
\BX$ be such that $\mathcal{N}^\infty (S)$ is dense in $X$. 
Then $S$ is in the norm closure of the generalised backward 1-shift operators  on $X$ 
so $T=\lambda +S \in \overline {\MiX}$ (the norm closure of $\MiX$) for each $\lambda\in\IC$ with $|\lambda|=1$.
\end{thm}

\begin{proof}
Let $\{x_n\}_{n=1}^\infty\subset \mathcal{N}^\infty (T)$ be a linearly independent
set whose span is dense in $X$.
We now apply the above construction to produce vectors
$\{ y^\ell_j : 1 \leq \ell < L + 1, 1 \leq j \le d_\ell \}$ which
satisfy the conditions
(\ref{eqn:shift-cond}) through
(\ref{eqn:ys-lin-indep}) for all $\ell < \infty$.

For each $(\ell, j)$  with $1 \leq j \leq d_\ell$,
we can use the Hahn-Banach theorem and 
choose $y^{\ell \, *}_j \in X^*$ (the dual space of $X$) so that
\[
	y^{\ell \, *}_j (y^m_k) = \delta_{\ell, n} \delta_{j,k}
	\mbox { if } m \leq \ell.
\]

Given any $\varepsilon > 0$, we
then choose $\varepsilon_\ell > 0$ such that
\[
	\sum_{\ell=1}^\infty \varepsilon_\ell \| y^{\ell +1 \, *}_1\| \|
	y^\ell_{d_\ell}\| < \varepsilon.
\]
We now define $S_\ell \colon X \to X$ for $\ell \geq 2$
by $S_\ell =0$ if
$y^{\ell -1, \, *}_1( S  y^\ell_{d_\ell}) \neq 0$ and
$S_\ell(x) = \varepsilon_\ell y^{\ell -1, \, *}_1(x)  y^\ell_{d_\ell}$
otherwise. Then $\|S_\ell\| \leq \varepsilon_\ell \| y^{\ell -1 \,
*}_1\| \| y^\ell_{d_\ell}\|$ and so $\sum_{\ell=1}^\infty \|S_\ell\|<
\varepsilon$. Put $S' = S + \sum_{\ell=2}^\infty S_\ell$.

We claim that $S'$ is then a generalised $1$-shift in the sense of 
Definition~\ref{def:genshift} with respect to the set obtained by
linearly ordering $\{ y^\ell_j : \ell \in \IN, 1 \leq j \leq d_\ell\}$
by taking those vectors in the order
\[
	y^1_{d_1}, \ldots,
	y^1_1, y^2_{d_2}, \ldots,
	y^2_{1}, y^3_{d_3}, \ldots.
\]
To verify this notice that for $(\ell, 1) \leq (k, j)$ we have
$S_\ell(y^k_j) \in E_{\ell-1}$ but
$S_\ell(y^k_j) =0$ if $\ell > k$.
Hence $S'(y^k_j) - y^k_{j+1}\in E_{k-1}$ if $j < d_k$ by
(\ref{eqn:shift-cond})
while
$S'(y^k_j) \in E_{k-1}$ if $j = d_k$
by
(\ref{eqn:shift-cond2}).
We have organised that 
$y^{\ell -1, \, *}_1( S'  y^\ell_{d_\ell}) \neq 0$ for $\ell \geq  2$,
which is the remaining condition needed for $S'$ to be a generalised
$1$-shift.

Since, for $|\lambda| =1$, we have $\lambda + S' \in \MiX$ by
Corollary~\ref{cor:bns}, and
$\|(\lambda +S) - (\lambda + S')\| < \varepsilon$ , we conclude that
$\lambda +S' \in \overline{\MiX}$ as required.
\end{proof}
\begin{cor}
\label{cor:nilpapprox}
Let $X$ be a separable infinite dimensional Banach space and let $N\in
\BX$ be a nilpotent.
Then $\lambda +N \in \overline {\MiX}$.
\end{cor}

We say that a sequence of vectors $\{x_n\}_{n=1}^\infty$ is \emph{minimal} if there exist functionals $x^*_n\in X^*$ such that 
$x^*_n(x_m)=\delta_{n,m}$ and a minimal sequence of vectors is \emph{fundamental} if its linear span is dense in $X$. 
The $x_n^*$ are known as biorthogonal functionals for $\{x_n\}_{n=1}^\infty$.
We call $\{x_n\}_{n=1}^\infty$ \emph{normalised} if $\|x_n\|=1$ for each $n$.
In particular a Schauder basis is a minimal fundamental sequence.

By forward shift with respect to a basis $\{x_n\}_{n=1}^\infty$ we mean an operator $S\in \BX$ satisfying $Sx_n=w_nx_{n+1}$ for $n\ge 1$, where $w_n\in\IC$ are weights.
We will also consider bilateral shifts when the basis is $\{x_n\}_{n=-\infty}^\infty$.

In the next Proposition we consider a more general type of operator akin to a mixture of shifts but we do not insist on a Schauder basis.

\begin{prop}
\label{lem:gen}
Let $X$ be a separable infinite dimensional Banach space.
Suppose $S\in B(X)$ is an operator such that 
$$Sx^i_j=a^i_jx^i_{j+1}\mbox{ and }Sy^i_j=b^i_jy^i_{j+1}$$ 
where $a^i_j,b^i_j\in \IK$ and 
$$\{x^i_j:1\le i< I_1+1, 1\le j<\infty\}\cup\{y^i_j:1\le i<I_2+1,-\infty< j<\infty\}$$ 
is a normalised minimal fundamental set for $X$ with biorthogonal functionals 
$$\{x^{i*}_j:1\le i< I_1+1, 1\le j<\infty\}\cup\{y^{i*}_j:1\le i< I_2+1, -\infty< j<\infty\}$$
where $0\le I_1,I_2\le \infty$ (but we can't have $I_1=I_2=0$).
 
If for each $i$ with $1\le i<I_1+1$ there is an increasing sequence $n^i_j$ with $|a^i_{n^i_j}|\|x^i_{n^i_j}\|\rightarrow 0$
and for each $i$ with $1\le i<I_2+1$ there is an increasing sequence $m^i_j$ with $|b^i_{m^i_j}|\|y^i_{m^i_j}\|\rightarrow 0$ 
then $\lambda +S$ is in the closure of the mixing operators.
\end{prop}
\begin{proof}
Let $\epsilon>0$ be given and let $\epsilon^i_j>0$,$\delta^i_j>0$ be such that 
$$\sum_{i=1}^{I_1}\sum_{j=1}^\infty\epsilon^i_j+\sum_{i=1}^{I_2}\sum_{j=1}^\infty\delta^i_j<\epsilon/2.$$
We have that $|a^i_{n^i_j}|\|x^{i*}_{n^i_j}\|\rightarrow 0$ so for each $i$ choose a subsequence $n^i_{k_j}$ such that $|a^i_{n^i_{k_j}}|\|x^{i*}_{n^i_{k_j}}\|<\epsilon^i_j$
and similarly choose a subsequence $m^i_{k_j}$ such that $|b^i_{m^i_{k_j}}|\|y^{i*}_{m^i_{k_j}}\|<\delta^i_j$.
Let $$K=\sum_{i=1}^{I_1}\sum_{j=1}^\infty a^i_{n^i_{k_j}}x^{i*}_{n^i_{k_j}}\otimes x^i_{n^i_{k_j}+1}+\sum_{i=1}^{I_2}\sum_{j=1}^\infty b^i_{m^i_{k_j}}y^{i*}_{m^i_{k_j}}\otimes y^i_{m^i_{k_j}+1}.$$ 
Then 
\begin{eqnarray*}
\|K\|&\le& \sum_{i=1}^{I_1}\sum_{j=1}^\infty |a^i_{n^i_{k_j}}|\|x^{i*}_{n^i_{k_j}}\|+\sum_{i=1}^{I_2}\sum_{j=1}^\infty |b^i_{m^i_{k_j}}|\|y^{i*}_{m^i_{k_j}}\|\\
&<&\sum_{i=1}^{I_1}\sum_{j=1}^\infty\epsilon^i_j+\sum_{i=1}^{I_2}\sum_{j=1}^\infty\delta^i_j<\epsilon/2
\end{eqnarray*}

\begin{quote}
\begin{claim}
$\mathcal{N}^\infty(S-K)$ is dense in $X$.
\end{claim}
\begin{proof}
First, we will show that $x^i_j\in \mathcal{N}^\infty(S-K)$ for each $i,j$.
Note that $(S-K)x^i_j=Sx^i_j=a^i_jx^i_{j+1}$ for $(i,j)\neq (i,n^i_{k_m})$.
We have $(S-K)x^i_{n^i_{k_j}}=a^i_{n^i_{k_j}}x^i_{n^i_{k_j}+1}-a^i_{n^i_{k_j}}x^i_{n^i_{k_j}+1}=0$.
Suppose $n^i_{k_j}< j\le m^i_k$, say $j=m^i_k-\ell$ , then  $(S-K)^{\ell+1}x^i_j=(a^i_ja^i_{j+1}\ldots a^i_{m^i_k-1})(S-K)x^i_{j+l})=0$.
So $x^i_j\in \mathcal{N}^\infty(S-K)$. 

Similarly $y^i_j\in \mathcal{N}^\infty(S-K)$ for $j\ge m^i_1$. 
But we also have that $(S-K)y^i_j=Sy^i_j$ for $j<m^i_1$. 
Let $j=m^i_1-\ell$ for some $\ell\in\IN$.
Then $(S-K)^{\ell+1}y^i_j=(b^i_jb^i_{j+1}\ldots b^i_{m^i_1-1})(S-K)y^i_{j+l}=0$ so $y^i_j\in\mathcal{N}^\infty(S-K)$.

Then $\SPAN\{x^i_j\}\cup\{y^i_j\}\subset \mathcal{N}^\infty (S-K)$ and therefore $\mathcal{N}^\infty(S-K)$ is dense in $X$.
\end{proof}
\end{quote}

Therefore by Theorem \ref{thm:genker} we have that there is a 1-shift $B$ such that $\|S-K-B\|<\epsilon/2$.
So $\|S-B\|\le \|S-(S-K)\|+\|(S-K)-B\|<\epsilon$ and therefore (via Corollary \ref{cor:bns}) $\lambda +S$ is in the closure of the mixing operators.
\end{proof}

\begin{cor}
Let $X$ be a separable infinite dimensional Banach space with a Schauder basis.
If $S$ is a weighted forward or bilateral shift on a normalised basis for $X$ with weights $w_j$ such that there exists an increasing sequence $n_j$ with $|w_{n_j}|\rightarrow 0$ 
then $\lambda +S$ is in the closure of the mixing operators.

In particular if $S$ is a quasinilpotent forward or bilateral shift on a normalised basis then $\lambda +S$ is in the closure of the mixing operators.
\end{cor}
\begin{proof}
Suppose $e_n$ is our normalised basis. By a remark at the beginning of \cite[1.b]{LindenstraussTzafriri} we have that $\sup_n\|e^*_n\|=K<\infty$.
Then $|w_{n_j}|\|e_n^*\|\le K|w_{n_j}|\rightarrow 0$ and the rest follows from Proposition \ref{lem:gen} by setting
$I_1=1$, $I_2=0$ and $x^1_j=e_j$ for the forward shift case and 
$I_1=0$, $I_2=1$ and $y^1_j=e_j$ for the bilateral shift case.

For the second part we need only show that if the shift is quasinilpotent then there must be a subsequence of the weights going to 0.
Assume for the sake of contradiction that this is not true.
Then there is $j\in\IN$ such that $C=\inf_{n\ge j}\{|w_n|\}\neq 0$.
Choose $0<\epsilon<C$.
Then there is $N\in\IN$ such that $\|S^nx\|<\epsilon^n\|x\|$ for all $n>N$.
$$C^n\le \|w_jw_{j+1}\ldots w_{j+n-1}e_{j+n}\|=\|S^ne_j\|<\epsilon^n<C^n$$
for $n>N$.

\end{proof}

\end{document}